\documentclass{rsepublic}
\usepackage[utf8x]{inputenc}

\usepackage{amsmath}
\usepackage{fancybox} 
\usepackage{multirow}
\usepackage{amssymb}
\usepackage{permute}  

\newtheorem{thm}{Theorem}[section]
\newtheorem{cor}[thm]{Corollary}

\newtheorem{conj}[thm]{Conjecture}

\theoremstyle{definition}
\newtheorem{defn}[thm]{Definition}
\newtheorem{example}[thm]{Example}

\newcommand{\Z}{\mathbb Z}

\newcommand{\note}[1]{}

\newcommand{\be}{\begin{enumerate}}
\newcommand{\ee}{\end{enumerate}}\usepackage{enumerate}

\begin{document}
\title[Distribution in Planar Nearrings]{Distribution and Generalized Centre in Planar Nearrings}
\author[T. Boykett]{Tim Boykett\\Institute for Algebra, Johannes Kepler University, 4040 Linz, Austria, and\\
Time's Up Research, Industriezeile 33b, 4020 Linz, Austria\\
tim@timesup.org, tim.boykett@jku.at}

\MSdates[`Received date']{`Accepted date'}

\label{firstpage}

\maketitle

\begin{abstract}
Nearrings are the nonlinear generalization of rings.
Planar nearrings play an important role in nearring theory, both from the structural side, being
close to generalized nearfields, as well as from an applications perspective, in geometry and combinatorial designs
related to difference families.
In this paper we investigate the distributive elements of planar nearrings.
If a planar nearring has nonzero distributive elements, then it is an extension of its zero multiplier part by an abelian group.
In the case that there are distributive elements that are not zero multipliers,
then this extension splits, giving an explicit description of the nearring. 
This generalizes the structure of planar rings.
We provide a family of examples where this does not occur, the distributive elements being precisely the zero multipliers.
We apply this knowledge to the question of determining the generalized center of
planar nearrings as well as finding new proofs of other older results.
\end{abstract}




\section{Introduction}

Nearrings are the nonlinear generalization of rings, having only one distributive law.
Planar nearrings are a special class of nearrings that generalize nearfields, themselves
a generalization of fields, which  play an important role in the structural
theory of nearrings\cite{wendtMatrix,wendtMin}, 
as well as having important geometric and combinatorial properties\cite{wendtAg,clay,KePilz}.

In this paper we look at the distributive elements of a planar nearring, in some sense the
ring-like elements of the planar nearring. 
This extends Aichinger's work on planar rings \cite{aichingerPlanarRings}.
We find that the general structure of planar nearrings with nontrivial distributive elements
generalizes the structure of planar rings.
We use this information to then investigate the generalized center of a planar nearrings, 
building on Farag, Cannon, Kabza and Aichinger's work \cite{AF04,CFK07}.
This was the original motivation for this work and can be found in  \S \ref{secGenCentre}.
As the results there indicate, it was necessary to obtain a good understanding of the
distributive elements of a planar nearring, which we undertake in \S \ref{secDist} and find
some small applications of in \S \ref{secApplications}.
The structure implied by nontrivial distributive elements forces several special forms
of planar nearrings, which we introduce in \S \ref{secNearVecSpaces} and use repeatedly throughout.

In the next section we introduce the necessary background about nearrings and planar nearrings.

\section{Background}

A \emph{(right) nearring} $(N,+,*)$ is an algebra such that
\begin{itemize}
 \item $(N,+)$ is a group with identity $0$
 \item $(N,*)$ is a semigroup
 \item for all $a,b,c\in N$, $(a+b)*c = a*c+b*c$ (right distributivity)
\end{itemize}
While it is readily seen that $0*a=0$ for all $a\in N$, it is not necessary that $a*0=0$ for all $a$. 
If this is the case, the nearring is called \emph{zero symmetric}.

A subset of $N$ is a \emph{subnearring} if it is closed as an additive group and a multiplicative semigroup.
A subnearring $I \subseteq N$ is a \emph{right ideal} if $I$ is a normal additive subgroup and for all $i\in I$, $n\in N$, $i*n \in I$.
A subnearring $I \subseteq N$ is a \emph{left ideal} if $I$ is a normal additive subgroup and for all $i\in I$, $n,m\in N$, $n*m - n*(m+i) \in I$.
A subnearring that is a left and right ideal is an \emph{ideal}.
As we expect, ideals correspond to the kernels of nearring homomorphisms.

We write $N^*$ for $N\setminus \{0\}$. If $(N^*,*)$ is a group, then $(N,+,*)$ is a \emph{nearfield},
generalizing fields by having a single distributive law.
Nearfields play an important role in geometries,
nearring and nearfield theory are well discussed in \cite{clay,ferrerobook,pilzbook,waehling}.
Nearfield were first described by Dickson using a process that is now called a Dickson nearfield.
All but 7 finite nearfields are Dickson.

The set $D(N) = \{n \in N \vert n(a+b)=na+nb\, \forall a,b\in N\}$  of \emph{distributive elements} 
of $N$ is the core investigation of this paper.

Elements $a,b\in N$ are called \emph{equivalent multipliers} if $x*a=x*b$ for all $x\in N$.
We write $a\cong b$, an equivalence relation.
A nearring  $(N,+,*)$ is called \emph{planar} if
\begin{itemize}
 \item The equivalent multiplier equivalence relation has at least 3 classes
 \item For every $a,b,c\in N$, $a\not\cong b$, the equation $x*a = x*b +c$ has a unique solution.
\end{itemize}

Every field other than $\Z_2$ is a planar nearring.
A finite nearfield (except $\Z_2$) is always planar, there are known to be nonplanar infinite nearfields \cite[page 46]{waehling}.
Planar nearrings are zero symmetric.

Planar nearrings can be described by fixed point free automorphism groups of groups.
Let $(N,+)$ be a group.
Then we say that some nonidentity $\phi \in Aut(N)$ is \emph{fixed point free} if for all $n\in N$, $n\phi=n$ iff $n=0$.
A group of automorphisms is fixed point free if all nonidentity elements are fixed point free.
Now let $\Phi \leq Aut(N)$ be a group of fixed point free automorphisms of $N$ acting from the right
such that $-id + \phi$ is bijective for all non identity $\phi \in \Phi$.
We write $a\Phi$ for the orbit containing $a$.
Let $R\subseteq N$ be a set of orbit representatives, $M\subseteq R$  a set.
Every element $a\in N$ can be written uniquely as $r_a\phi_a$ for some $r_a \in R$, $\phi_a\in\Phi$.
We define a multiplication by:
\begin{equation}
 a*b =
 \begin{cases}
  0 & r_b \in M  \\
  r_a(\phi_a\phi_b)=a\phi_b\,&r_b \not\in M
 \end{cases}
\end{equation}
Then $(N,+,*)$ is a planar (right) nearring.
All planar nearrings can be so derived\cite{clay}.

Let $a\Phi^* = a\Phi \cup \{0\}$.
We call $M$ the set of \emph{zero multipliers}.
If $a \in M$ then we call $a\Phi$ a \emph{zero multiplier orbit}, and we will call all elements of this orbit zero multipliers.
Then $n\in N^*$ is a zero multiplier iff $n \in M\Phi^*$ iff $a*n=0$ for all $a\in N$.
Note that $a*b=0$ iff $a=0$, $b=0$ or $r_b\in M$. In particular, $n*b=0$ for all $n$ iff $b=0$ or $r_b\in M$.
The elements $r\in R\setminus M$ are right identities, $x*r = x$ for all $n\in N$.
A planar nearring has a left identity iff it has an identity iff it has exactly one nontrivial $\Phi$ orbit iff it is a planar nearfield.
We use $Z(\Phi)$ to denote the centre of the group $\Phi$ and remark that we will use the British spelling throughout this paper.
The distributive elements of a nearfield are called the \emph{kern} of the nearfield and contains the multiplicative centre.

There are related nearrings. 
The trivial nearring on any group $(N,+)$ with multiplication $a*b=0$ would correspond to $\cong$ having one equivalence class.
The Malone trivial nearrings \cite{CFKN16} correspond to $\Phi$ having order 1, so $R=N\setminus \{0\}$, 
$a*b =0$ if $b\in M$ and $a*b=a$ otherwise.
It is interesting to note that the complemented Malone nearrings in \cite{CFKN16} are a generalization of 
planar nearrings with $\Phi$ of order 2, allowing fixed points (i.e. elements of additive order 2) that are zero multipliers.

In the next section we will look at two special constructions for planar nearrings.
Then in the following section we look at the distributive elements of a planar nearring.
With that knowledge, we will determine the generalized centre of a planar nearring.

\section{(Near) Vector Spaces}
\label{secNearVecSpaces}

In this section we  construct two families of example planar nearrings, which will prove to be
useful in the rest of the paper.

Let $V$ be a vector space over a division ring $D$ of order at least 3 and $\phi: V\rightarrow D$ a vector space epimorphism
which acts from the right.
Define a multiplication $*: V\times V \rightarrow V$ by $a*b = a(b\phi)$.
By \cite[theorem 4.1]{aichingerPlanarRings} and \cite[theorem 5.2.1]{wendtDiss} this is a planar ring and all planar rings have this form.
Using the terminology above, $\Phi$ is isomorphic to the nonzero elements of $D$ under multiplication,
$R\setminus M = \phi^{-1}(1)$ and
the elements of $M$ can be chosen arbitrarily from the orbits that lie completely within $\ker \phi$.
Let $v_1\in R$ be arbitrary. 
If $V$ is a finite dimensional vector space, we can choose a new basis $v_1,\dots,v_n \in V$ such that 
$v_2,\dots,v_n\in \ker \phi$ and for all $x=(x_1,\dots,x_n) \in V$, $x\phi  = x_1$.

We can generalize this construction to nearvector spaces\footnote{We note in passing the existance
of another, similar, definition of a nearvector space used by Karzel and colleagues\cite{karzel84,karzelkist84} in which the right
nearfield scalars operate from the left, giving significantly different properties.}.
We begin with a brief overview and some definitions of nearvector spaces. 
See \cite{Andre} for further details.

\begin{defn} A pair $(V,A)$ is called a \emph{nearvector space} if:
\be
\item $(V,+)$ is a group and $A$ is a set of endomorphisms of $V$, which act from the right;
\item $A$ contains the endomorphisms $0$, {\it id} and $-${\it id};
\item $A^*=A\setminus\{0\}$ is a subgroup of the group Aut$(V)$;
\item $A$ acts fixed point freely on $V$;
\item the quasi-kernel $\{x\in V \,|\, \forall \alpha,\beta\in A, \exists\gamma\in A : x\alpha + x\beta = x\gamma\}$ generates $V$ as a group. 
\ee
\end{defn}

We sometimes refer to $V$ as a {\it nearvector space over $A$}. 
We write $Q(V)$ for the quasi-kernel of $V$.
The elements of $V$ are called {\it vectors} and the members of $A$ {\it scalars}
and it turns out that $A$ is a nearfield. 
The action of $A$ on $V$ is called {\it scalar multiplication}. 
Note that  $(V,+)$ is an abelian group. 
Also, the dimension of the nearvector space, $\dim(V)$, 
is uniquely determined by the cardinality of an independent generating set for $Q(V)$.

In  \cite[theorem 3.4]{vanderWalt92} we find a characterization of finite dimensional nearvector spaces, 
see also  \cite[theorem 4.6]{Andre}.

\begin{thm} Let $V$ be a group and let $A:=D \,\cup\, \{0\}$, where $D$ is a fixed point free group of automorphisms of $V$. 
Then $(V,A)$ is a finite dimensional nearvector space if and only if there exists a finite number of nearfields, 
$F_{1},F_{2},\ldots,F_{n}$, semigroup isomorphisms $\psi_{i}:A \rightarrow F_{i}$ 
and a group isomorphism $\Phi:V \rightarrow F_{1}\oplus F_{2}\oplus \dots \oplus F_{n}$ 
such that if \[\Phi(v) = (x_{1},x_{2}, \dots, x_{n}),\,\,\,(x_{i} \in F_{i})\]
then
\[\Phi(v\alpha) = (x_{1}(\alpha\psi_{1}), x_{2}(\alpha\psi_{2}), \dots, x_{n}(\alpha\psi_{n})),\]
for all $v \in V$ and $\alpha \in A$.
\end{thm}

In \cite[4.13 ff]{Andre} we find the following.
A nearvector space is \emph{regular} if all the $\psi_i$ are identical (up to nearfield automorphisms).
Every nearvector space has a unique maximal decomposition $V= V_1\oplus V_2\oplus\dots$ into regular
sub-nearvector spaces $V_i$.
Then for all nonzero $u\in Q(V)$, there is precisely one $i$ auch that $u \in V_i$.
Note that a regular vector space over a field is a vector space.

We can construct a planar nearring from a near vector space along the lines used above for vector spaces. 
Let $V$ be a nearvector space over a  nearfield $F$ of order at least 3.
Let $\phi:V\rightarrow F$ be a nearvector space epimorphism and define $a*b = a(b\phi)$.
The right identities $R\setminus M$ are $\phi^{-1}(1)$ and the representatives in $M$ can be chosen arbitrarily.


\begin{example}
 Let $V$ be the two dimensional nearvector space over $F=\Z_5$ with $\psi_1$ 
 the identity and $\psi_2 = (2\,3)$ the automorphism of $F^*$
 exchanging $2$ and $3$, equivalently $x\psi_2= x^3$.
 Note that $Q(V) = F\times\{0\} \cup \{0\} \times F$.
 Taking $\phi(v_1,v_2) = v_1$,
 we  obtain a planar nearring $(V,+,*)$.
 Then $(v_1,v_2)\in D(V)$ iff $v_1 x + v_1y = v_1(x+y)$ and $v_2(x\psi_2) + v_2(y\psi_2) = v_2((x+y)\psi_2)$ for all $x,y,\in F$.
 The first equation always holds, but the second equation can be seen to fail for $x=y=1$ unless $v_2=0$. Thus we see that
 $D(V) = F\times\{0\}$.
\end{example}

\begin{example}
 Let $V$ be the two dimensional nearvector space over $F$, the proper nearfield of order 9 with kern $K$ of order 3.
 Let $\psi_1=\psi_2$ be the identity.
 Then $Q(V) = K\times K$.
 Taking $\phi(v_1,v_2) = v_1$, we  obtain a planar nearring $(V,+,*)$.
 Then $(v_1,v_2)\in D(V)$ iff $v_1 x + v_1y = v_1(x+y)$ and $v_2 x  + v_2 y = v_2(x+y)$ for all $x,y,\in F$.
 These equations hold iff $v_1$ and $v_2$ are both in the kern $K$ of $F$, so $D(V) = K \times K$.
\end{example}


\begin{conj}
Let $F$ be a nearfield with kern $K$,
 let $V$ be a finite dimensional $F$-nearvector space derived planar nearring as above,
 $V=V_1\oplus\dots\oplus V_n$ the regular decomposition with $V_i = F^{n_i}$.
 Then $D(V)=D(V_1)\oplus\dots\oplus D(V_n)$ with $D(V_i) = K^{n_i}$. 
\end{conj}

It is worth noting in passing that nearvector spaces and the homogeneous mappings of them to themselves
are closely related to questions about nearring matrices over the associated nearfield.
Thus we hope that future work here could shed light on the question raised in the final section of
\cite{wendtMatrix} as to the inverses of units in matrix nearrings over planar nearfields.

\section{The Distributive Elements}
\label{secDist}

In this section we investigate the distributive elements of a planar nearring.
Some examples are well known. A finite field is a planar nearfield and thus a planar nearring,
with all elements being distributive. In \cite{aichingerPlanarRings} (see \S \ref{secNearVecSpaces} above) the structure of
planar rings is completely determined, so we know what happens when $D(N)=N$.
The distributive elements of a nearfield are called the \emph{kern} of the nearfield.

%

Using Sonata \cite{sonata} we found all planar nearrings up to order 15 with nontrivial $D(N)$.
\begin{enumerate}
 \item The fields of order 3,4,5,7,8,9,11,13.
 \item The proper nearfield of order 9 with kern of order 3.
 \item The planar ring of order 9.
 \item An example of order 9. The additive group is $\Z_9$, $\Phi=\{1,-1\}$, $R=\{2,3,5,8\}$, $M=\{3\}$.
 The distributive elements are the zero multipliers $\{0,3,6\}$.
 \item An example of order 15, $\Phi$ of order 2 with generator $g$ acting 
 on $\Z_3 \times \Z_5$ as $(x,y)*g = (-x,-y)$, the orbit $\{0\}\times \Z_5$ zero multipliers. The distributive
 elements are $\Z_3 \times \{0\}$.
\end{enumerate}
We see that 1-3 can be readily explained, but not 4 or 5.
One of the goals of this paper is to understand all these examples in terms of general classes.

We first determine some properties of the orbits that contain distributive elements.

\begin{lemma}
\label{lemmaDistAddClosed}
Let $d\in N$. Then $d \in D(N) \Rightarrow d\Phi^*$ is additively closed.
\end{lemma}
\begin{proof}
Suppose $d\Phi^*$ is not additively closed, so that there exist some $\phi_1,\phi_2$ such that $d\phi_1 + d\phi_2 \not\in d\Phi^*$.
Let $r_3\phi_3 = r_d\phi_1 + r_d\phi_2$. 
Then $d*(r_d\phi_1 + r_d\phi_2) = d * r_3\phi_3 = r_d\phi_d\phi_3 \in d\Phi^*$.
However $d*r_d\phi_1 + d*r_d\phi_2 = d*\phi_1 + d*\phi_2 \not \in d\Phi^*$, so $d$ is not distributive, a contradiction.
\end{proof}

\begin{lemma}
\label{lemmaDistMultClosed}
Let $N$ be a planar nearring.
Let $d\in D(N)$ be a non zero multiplier, $d=r_d\phi_d$. 
Then $ \{\phi\in\Phi \vert r_d\phi \in D(N)\} \leq \Phi$ is a subgroup, containing $Z(\Phi)$.
\end{lemma}
\begin{proof}
From  lemma \ref{lemmaDistAddClosed} above, we know that $d\Phi^*$ is additively closed.

Let $a,b\in N$, $\phi\in \Phi$. Let $r=r_d$.
Then we can show the following.
\begin{align}
 r \phi * (a+b) &= r \phi * r \phi_d^{-1} * r\phi_d * (a+b)\\
  &=  r \phi * r \phi_d^{-1} *(r\phi_d * a + r\phi_d*b) \\
  &=  r \phi * r \phi_d^{-1} *(r\phi_d * r*a + r\phi_d*r*b) \\
  &=  r \phi * r \phi_d^{-1} *r\phi_d * (r*a + r*b) \\
  &=  r \phi * (r*a + r*b) 
\end{align}

We use this in the following calculation.
We know that $r_d \Phi^*$ is additively closed and that $r=r_d$ is  a left multiplicative identity in $r_d\Phi$. 
\begin{align}
 r \phi_d^{-1} * a + r \phi_d^{-1} * b &= r * (r \phi_d^{-1} * a + r \phi_d^{-1} * b) \\
   &= r\phi_d^{-1} * r\phi_d *  (r \phi_d^{-1} * a + r \phi_d^{-1} * b) \\
   &= r\phi_d^{-1} *   (r\phi_d *r \phi_d^{-1} * a + r\phi_d *r \phi_d^{-1} * b) \\
   &= r\phi_d^{-1} *   (r * a + r * b) \\
   &= r\phi_d^{-1} *   (a +  b) 
\end{align}
Thus the multiplicative inverse of $d$ in $d\Phi$  is in $D(N) \cap d\Phi$. 
By standard arguments, $D(N) \cap d\Phi$ is multiplicatively closed, thus a group,
so$ \{\phi\in\Phi \vert r_d\phi \in D(N)\} \leq \Phi$ is a subgroup of $\Phi$.

Now we know that $r_d*(a+b)=r_d*a+r_d*b$. 
Let $\phi \in Z(\Phi)$. Then let $a,b\in N$, 
\begin{align}
 r_d\phi *(a+b) &= r_d\phi \phi_{a+b} \\
  &= r_d\phi_{a+b}\phi \\
  &= r_d(a+b)\phi \\
  &= r_d\phi*a + r_d\phi*b
\end{align}
so $Z(\Phi) \leq \{\phi\in\Phi \vert r_d\phi \in D(N)\}$ and we are done.
\end{proof}

\begin{lemma}
\label{lemmaNearfield}
Let $d\in N$ be a non zero multiplier. Then $d \in D(N) \Rightarrow d\Phi^*$ is a planar nearfield.
\end{lemma}
\begin{proof}
From lemma \ref{lemmaDistAddClosed} above, we know that $d\Phi^*$ is additively closed.
Thus $d\Phi^*$ is additively and multiplicatively closed, forming a planar subnearring. 
We note that there is precisely one orbit on this nearring, so the planar nearring must be a nearfield
with $r_d$ the multiplicative identity.

Let $n,c\in D(N)$ be arbitrary, then there exists a unique $x\in N$ such that $x-x*n=c$ by
the planarity of $N$, but this might not be in $d\Phi^*$. However $r_d*x - r_d*x*n = r_d*(x-x*n) = r_d*c=c$
so $r_d*x$ is also a solution to the equation. This solution is unique so $x=r_d*x\in d\Phi^*$ and
$d\Phi^*$ is a planar nearfield.
\end{proof}

Similarly we know that, even in the case that there are distributive elements that are zero multipliers,
additive closure of the orbit $d\Phi^*$ allows us to define a multiplication $d\phi_1 \circ d\phi_2 = d(\phi_1\phi_2)$
that gives us $(d\Phi^*,+,\circ)$ a planar nearfield.

Thus we know a lot more about the forms of $\Phi$ that can emerge, as not all fixed point free 
automorphism groups arise as the multiplicative group of a nearfield.


\begin{lemma}
\label{lemmaEquivMult}
 Let $N$ be a planar nearring with nontrivial distributive elements.
 Then for every $m\in M\Phi^*$, $a\in N\setminus M\Phi^*$, $\phi_{m+a} = \phi_a$.
\end{lemma}
\begin{proof}
We proceed by calculation. Let $0\neq d \in D(N)$.
\begin{align}
 r_d\phi_d\phi_{m+a} &= d * (m+a) = d*m+d*a \\
 &= 0+d*a = d*a = r_d\phi_d\phi_a\\
 \Rightarrow \phi_d\phi_{m+a} &= \phi_d\phi_a \\
 \Rightarrow \phi_{m+a} &= \phi_a
\end{align}
By symmetry, $\phi_{a+m} = \phi_a$ as well.
\end{proof}


\begin{lemma}
\label{lemmaHomo}
 Let $N$ be a planar nearring. If $D(N)$ is nontrivial, then the zero multipliers form a nearring ideal.
\end{lemma}
\begin{proof}
Let $0\neq d \in D(N)$.
Let $K = M\Phi^*$, the zero multipliers.
The mapping $\rho : n\mapsto d*n$ is an additive homomorphism.
Then $n\in \ker \rho$ iff $d*n=0 $ iff $n\in K$ so the zero multipliers form an additively normal group.

Let $m,n \in N$, $k\in K$.
If $n$ is a zero multiplier, then $k*n=0$ so $k*n \in K$.
If $n$ is not a zero multiplier, then $k*n = r_k\phi_k\phi_n \in K$, so $K$ is a right ideal.
If $n+k \in K$, then since $K$ is a subgroup, $n \in K$ so  $m*n - m*(n+k) = 0-0 \in K$.
If $n+k \not\in K$, then remember that by the lemma above, $\phi_{n+k}=\phi_n$. Then
\begin{align}
 m*n - m*(n+k) &= m*n - r_m\phi_m\phi_{n+k}\\
  &= m*n - r_m\phi_m\phi_n \\
  &= m*n-m*n = 0 \in K
\end{align}
so we have a left ideal and thus an ideal.
\end{proof}

Note that, in general, the mapping $\rho$ is not a nearring homomorphism, unless $\phi_d$ is the identity.
This can only be guaranteed to be the case when $D(N)$ contains non zero multipliers, by lemma  \ref{lemmaDistMultClosed}.

This implies that $(N,+)$ is an extension of  $(\rho(N),+)$ by $(K,+)$.
By lemma \ref{lemmaNearfield} and the comments afterwards, we know that $(\rho(N),+)$ is
the additive group of a nearfield, thus abelian and, in the finite case, elementary abelian.
When all distributive elements are zero multipliers, we do not necessarily have that the extension splits.
If we have a non zero multiplier distributive element, then we get a clear result.

\begin{thm}
\label{thmsemidirect}
Let $N$ be a planar nearring with automorphism group $\Phi$.
Let $d\in D(N)$ be a non zero multiplier. 
Then  there is a subnearfield $F\leq N$ with $F^*\cong \Phi$ 
and an additive group $K$ with $\Phi$ a group of fixed point free automorphisms, such that
$N \cong K \rtimes F$ as an additive group and
\begin{align}
  (a,b) * (c,d) &= \begin{cases}
                   0 & d=0 \\
                   (a\phi_d,b\phi_d) & \mbox{otherwise}
                  \end{cases}
\end{align}
such that $(N,+,*) \cong (K \rtimes F,+,*)$.
The zero multipliers form an ideal $K \times \{0\}$.
\end{thm}
\begin{proof}
By the previous lemma, we know that $(N,+)$ is an extension of $(K,+)$ by $(\rho(N),+)$.
By lemma \ref{lemmaNearfield} we know that $\rho(N)$ is a nearfield, let $F=\rho(N)$.
Because $(F,+)$ is a subgroup of $N$ that is fixed by $\rho$, the extension splits,
so $(N,+) \cong (K,+) \rtimes (F,+)$.
Let $k\in K$, $f\in F$. Writing $k^f = f+k-f$ we know that for all
$k_1,k_2\in K$, $f_1,f_2\in F$, $(k_1,f_1) + (k_2,f_2) = (k_1 + k_2^{f_1},f_1+f_2)$.

We can write each element of $N$ as $(k,f) = (k,0)+(0,f) \in M\Phi^* + F$ so by lemma \ref{lemmaEquivMult}
we know that $\phi_{(k,f)} = \phi_{(0,f)}$. We write $\phi_f$ for $\phi_{(0,f)}$.
Then we can write the multiplication on $K \rtimes F$ as above.
The representatives for $K\rtimes F$   are $\{(k,1) \vert k\in K\} \cup \{(m,0) \vert m\in M\}$.
Since the additive groups are isomorphic as $\Phi$-groups and the representatives 
are matched by the isomorphism, we know that the resulting planar nearrings are isomorphic.
\end{proof}


We see that the example on additive group $(\Z_9,+)$ falls outside this theorem.
The distributive elements lie within the zero multipliers and the additive group is not a semidirect product of the
zero multipliers with anything.

%
%

We can create a family of examples of planar nearrings with $D(N)$ lying within the
zero multipliers, based upon the example on page 49 of \cite{clay}.
These examples do not split.

\begin{example}
 Let $p$ be an odd prime,  $N=\Z_{p^2}$ the cyclic group of order $p^2$. 
There is a cyclic subgroup $\Phi$ of the multiplicative semigroup of order $p-1$.
One of the orbits of this automorphism group is $p\Z_{p^2}$. 
These are our zero multipliers, this orbit has representative $p\in \Z_{p^2}$. 
We choose the rest of our representatives to be a coset of $p\Z_{p^2}$.
Then the resulting planar nearring has $D(N) = p\Z_{p^2}$, all zero multipliers.
\end{example}

We know (e.g.\ \cite{mayrthesis}) that the additive group of a finite planar nearring is nilpotent and thus a
direct sum of $p$-groups.
Thus  a finite planar nearring is a finite direct sum of planar nearrings of prime power order.
Thus by lemma \ref{lemmaNearfield} at most one of these summands has a non zero multiplier distributive element.
If one summand has  such an element, then lemma \ref{lemmaHomo} indicates that we have a trivial multiplication of all
summands other than the one with a non zero multiplier distributive element.
We have shown the following.
\begin{cor}
 Let $N$ be a finite planar nearring with nontrivial distributive elements.
 Let $\Phi$ be the multiplicative group associated to $N$, $R$ the representatives and $M$ the zero multipliers.
 Then $(N,+)$ is the direct sum of finitely many p-subgroups $N=N_1 \oplus\dots\oplus N_k$, 
 only one of which has a nontrivial multiplication,
 so $R\setminus M \subseteq \{0\} \oplus \dots \oplus\{0\}\oplus N_i\oplus \{0\}\oplus  \dots \oplus \{0\}$ for some $i$.
\end{cor}

In the example of order 15, we have the $\Z_3$ as the nearfield with automorphism group of order 2 and $\Z_5$ having the same automorphism group generated by $(1\,4)(2\,3)$ acting on it. 
Then $N=\Z_3 \times \Z_5$ with $\{0\}\times \Z_5$ forming the zero multipliers in the nearring.

Thus we have shown all of our small examples of planar nearrings with nontrivial distributive elements fall into 
larger classes of examples.

\section{Some applications}
\label{secApplications}

In this section we look at the way that these results can be contextualized in relation to
other similar results.

%

One of the strong applications of planar nearrings is in the construction of BIBDs.
The blocks of the BIBD are $\{a\Phi^*+b\vert 0\neq a \in N, b \in N\}$ and 
the basic blocks are $\{a\Phi^*\vert 0\neq a \in N\}$.
The following result shows that the construction used  in \cite{clay72}
is the only way that all basic blocks can be subgroups.

\begin{lemma}
\label{lemmaAllOrbitsAddClosed}
 Let $N$ be a finite abelian planar nearring with more than one orbit, 
 in which all orbits $a\Phi^*$ are additively closed.
 Then $N$ is a  vector space over a subfield of $N$.
\end{lemma}
\begin{proof}
Let $F = a\Phi^*$ for some non zero multiplier $a$.
$F$ is additively and multiplicatively closed, thus a planar subnearring.
$F$ contains only one orbit of $\Phi$, so it has an identity and is thus a
nearfield.
Thus $F$ is prime power order. If $F$ is odd, then $\Phi$ is even order, thus $-1\in \Phi$.
If $F$ is even, then $-1=1\in \Phi$.

We see that $\Phi^*$ acting on $(N,+)$ satisfies the first four conditions for being a nearvector space.
By the additive closedness of each $n\Phi^*$, for each $\alpha,\beta \in \Phi^*$, $n\alpha + n\beta \in n\Phi^*$
so there exists some $\gamma\in\Phi^*$ such that $n\alpha + n\beta = n\gamma$.
Thus $Q(N) = N$ and $N$ is a near vector space over $\Phi^*$.
We see that $F$ must then be the nearfield from van der Walt's result above, $(\Phi^*,*) \cong (F,*)$.
By \cite[Satz 5.5]{Andre} we know that $Q(N)=N$ implies that $N$ is a vector space and that $F$ is actually a field.
\end{proof}

We note in passing that if a planar nearring has non trivial distributive elements, then
we know that the corresponding orbits give additively closed basic blocks.
Thus by \cite[Thm 7.17]{clay} we will obtain a statistical, but not a geometric BIBD.

We also obtain Aichinger's result as a corollary.

\begin{cor}
 Let $N$ be a planar ring. Then $N$ is a vector space over a field $F$ with $\Phi = F^*$.
\end{cor}
\begin{proof}
All elements of $N$ are distributive, so we know by lemma \ref{lemmaDistAddClosed} that every orbit is additively closed.
Select some non zero multiplier $d\in N$ and we see that $d*N$ is a distributive nearfield, that is, a field.
All orbits are additively closed, so we obtain a nearvector space with $Q(N)=N$, so by the same argument as above,
we know that $N$ is a vector space.
\end{proof}

\section{The Generalized Centre}
\label{secGenCentre}

Now that we know some things about the distributive elements of a planar nearring, we can 
say some things about the generalized centre.

Let $(N,+,*)$ be a 0-symmetric nearring. 
The \emph{generalized centre} of $N$ is $GC(N) = \{n\in N \vert nd=dn\; \forall d\in D(N)\}$
\cite{AF04,CFK07}.

The generalized centre was introduced because the centre of a nearring is not always well behaved.
For instance, it is not always a subnearring, while the generalized centre is.
If $N$ is a ring, then $GC(N)$ is the usual centre of $N$.

The generalizes centre of a planar nearfield $F$ is the set of elements that commute with the kern.
In the finite case, the kern is the multiplicative centre and thus $GC(F) = F$.
In the infinite case, when the kern is distinct from the multiplicative centre, we know
only that $GC(F)$ contains the multiplicative centre.

\begin{thm}
 Let $N$ be a planar nearring. Then $GC(N)$ is one of four cases:
 \begin{enumerate}
  \item If $D(N)$ intersects only zero multiplier orbits, then $GC(N)$ is the zero multipliers, an ideal.
  \item If $D(N)$ intersects more than one orbit of $\Phi$ and at least one of them is not a zero multiplier, then $GC(N)=\{0\}$.
  \item If $D(N)$ intersects exactly one orbit of $\Phi$, which is not a zero multiplier, let it be $a\Phi$. 
  Then $aZ(\Phi)^* \leq GC(N) \leq a\Phi^*$.
  \item If $D(N)=\{0\}$, then $GC(N) = N$.
 \end{enumerate}
\end{thm}
\begin{proof} We proceed by cases.

Case 1: Suppose $D(N)$ intersects only zero multiplier orbits.
Then for all $d\in D(N)$, for all $n\in N$,  $nd=0$. Thus if $r_n \in M$, $dn=0$ and thus $n \in GC(N)$. 
So $GC(N) = \cup \{ a\Phi^* \vert a \in M\}$, the zero multipliers, which we know from lemma \ref{lemmaHomo} to be an ideal.

Case 2: Suppose $D(N)$ intersects two  orbits nontrivially, one of them is a non zero multiplier orbit.
Let $a,b\in D(N)$ be in distinct orbits, $a$ not a zero multiplier.
Then a nonzero $c\in GC(N)$ implies that $ca=ac$, so $r_a=r_c$ and $c$ is not a zero multiplier. 
Thus $cb=bc$ implies that $b$ is also not a zero multiplier, so $r_c=r_b$, but $r_a\neq r_b$ so we have a contradiction,
so $GC(N)$ is trivial.

Case 3: Let $c\in GC(N)$, so $cd=dc$ for all $d\in D(N)$. Then $r_c=r_d$ so $c \in a\Phi^*$, giving us the upper bound.
We know that $F := a\Phi^*$ is a nearfield, so $D(N)=K$ is the kern of $F$.
If $K^* \mathrel{\unlhd} F^*$ as multiplicative groups, then $K=Z(F)$ by  \cite{Andre1963},
so $GC(N)=F=a\Phi^*$ showing that this bound can be achieved, for instance in the finite case.
Otherwise $\phi_c \in Z(\Phi)$  gives us the lower bound.

Case 4: If $D(N)$ is trivial, then by zerosymmetry, $GC(N)=N$.
\end{proof}

We can break this down depending upon the properties of the additive group.

\begin{cor}
 Let $N$ be a planar nearring with nonabelian additive group.
  Then $GC(N)$ is one of four cases:
 \begin{enumerate}
  \item if $D(N)$ intersects only zero multiplier orbits, then $GC(N)$ is the zero multipliers, an ideal.
  \item If $D(N)$ intersects more than one orbit of $\Phi$ and at least one of them is not a zero multiplier, then $GC(N)=\{0\}$.
  \item If $D(N)$ intersects exactly one orbit of $\Phi$, which is not a zero multiplier, let it be $a\Phi$, then $GC(N) =D(N)= a\Phi^*$.
  \item If $D(N)$ intersects no nonzero orbit of $\Phi$, then $GC(N) = N$.
 \end{enumerate}
\end{cor}
\begin{proof} 
Only the third case is different to the above.
If $N$ is additively nonabelian, then we know that the fixed point free automorphism group
is of odd order, thus cyclic, see e.g.\ \cite{mayrthesis}.
Thus $Z(\Phi)=\Phi$, so $D(N)$ is all of one orbit with zero.
Because the multiplication within this orbit is commutative, the generalized centre is all of the orbit and we are done.
\end{proof}

If the additive group is abelian, then many interesting and strange things can happen with skew fields.
However in the finite case we know more.

\begin{cor}
 Let $N$ be a finite planar nearring with abelian additive group.
  Then $GC(N)$ is one of four cases:
 \begin{enumerate}
  \item If $D(N)$ intersects only zero multiplier orbits, then $GC(N)$ is the zero multipliers, an ideal.
  \item If $D(N)$ intersects more than one orbit of $\Phi$ and at least one of them is not a zero multiplier, then $GC(N)=\{0\}$.
  \item If $D(N)$ intersects exactly one orbit of $\Phi$, which is not a zero multiplier, let it be $a\Phi$. Then $GC(N) = a\Phi^* \geq D(N)$.
  \item If $D(N)$ intersects no nonzero orbit of $\Phi$, then $GC(N) = N$.
 \end{enumerate}
\end{cor}
\begin{proof}
Only the third case is different to the above.
We note that abelian addition implies that the distributor $D(N)$ is additively closed.
Thus $D(N)$ is a planar ring.
By \cite{aichingerPlanarRings} we know that a planar ring is derived from a vector space over a field, where the field multiplication
is isomorphic to the fixed point free automorphisms of the planar ring.
We know that $D(N)$ lies within one orbit which is a nearfield, so $(D(N)^*,\cdot)$ is a cyclic subgroup
of $\Phi$ and thus  precisely the centre of $\Phi$.
Thus all elements of the orbit containing the distributor are in the generalized centre.
\end{proof}

\section{Conclusion}

In this paper we have investigated the distributive elements in a planar nearring.
We have been able to show that if there are nontrivial distributive elements then the additive
group is an extension of an abelian subgroup by the zero multipliers.
This additive group is the additive group of a nearfield, so elementary abelian in the finite case.
If the distributive elements include non zero multipliers, then the extension splits and
we obtain a clear structure.

As a result, we are able to re-prove Aichinger's theorem on planar rings as a corollary,
as well as Clay's results on BIBDs with additively closed basic blocks.
It is unclear whether lemma \ref{lemmaAllOrbitsAddClosed} can be extended to the infinite case.
Applying these results to the question of the generalized centre, we are able to obtain a clear
set of cases and to describe the generalized centre.

It would be valuable to know what sort of other examples can occur with the distributive elements
all lying within the zero multipliers, in order to complete the classification of structures.

It would be of value to calculate $D(N)$ explicitly in theorem \ref{thmsemidirect}.
It is easy to see that $D(N)$ is a direct product of some subset $E\subseteq K$ and the kernel of $F$.
The question is how to calculate which parts of $K$ have $a(\phi_{b+c}) = a\phi_b + a\phi_c$ 
where the first addition is in $F$ while the second is in $K$.
Note that the orbits will be  additively closed, giving us nearfields. This
might be another  nearvector space construction.

As we have been able to determine the generalized centre of planar nearrings, we can now look forward to describing the 
generalized centre of more complex classes of nearrings.

\section{Acknowledgements}
 Research supported by SFB Project F5004 of the Austrian Science Foundation, FWF.
 I would like to thank my colleagues G\"unter Pilz and Wen-Fong Ke for some 
 insightful questions in the early development of this paper.



\label{lastpage}

\end{document}